\documentclass[10pt,A4paper,reqno]{amsart}
\usepackage{amssymb}
\usepackage{graphicx}
\usepackage{amscd}
\usepackage{amsmath}
\usepackage{amsfonts}
\usepackage{mathrsfs}
\providecommand{\U}[1]{\protect\rule{.1in}{.1in}}
\providecommand{\U}[1]{\protect\rule{.1in}{.1in}}
\providecommand{\U}[1]{\protect\rule{.1in}{.1in}}

\newtheorem{example}{Example}[section]
\newtheorem{theorem}{Theorem}[section]
\newtheorem{defin}{Definition} [section]
\newtheorem{defins}{Definitions} [section]
\newtheorem{propo}{ Proposition}[section]

\newtheorem{remark}{Remark} [section]
\newtheorem{remarks}{Remarks} [section]
\numberwithin{equation}{section}
\begin{document}%\color{red}
\title[]{On the Blow-up criterion of Navier-Stokes equation associated with the Weinstein operator}
\author[Y. Bettaibi ]{}
\maketitle
\vskip 1cm
\centerline{\bf Youssef Bettaibi}
\centerline{E-mail : youssef.bettaibi@yahoo.com}
\centerline{ University of Gabes, Faculty of Sciences of Gabes, LR17ES11 Mathematics }\centerline{and Applications, 6072, Gabes, Tunisia
}

\vskip 1cm

\begin{abstract}
In this paper we give Navier-Stokes system
 associated with the Weinstein operator $(NSW)$
  (see Eq.\eqref{11}), We study the existence and uniqueness of solutions to equations (NSW) in  $L_{\alpha}^{p}\left(\mathbb{R}_{+}^{d+1}\right), 2 \alpha+d+2<p \leq \infty$, and we proved  some properties of the maximal solution of  equation. If the maximum time $T^*$is finite, we establish that the growth of $\left\| u ( {t}) \right\|_{L ^ {p}_{\alpha}} $ is at least of the order of ${\left(T^{*}-t\right)^{{-\frac{2 p}{p-2 \alpha-d-2}}}},$  fo rall $t$ in $\left[0, T^{*}\right]$, also we give some 
      blow-up results. 
\end{abstract}

{\bf 2010 AMS Classification : } {76D05; 42B10; 42B30; 42B35;35-XX; 35B44; 35Q30}

\section{\textbf{Introduction}}
Most physical systems are modeled by nonlinear partial differential equations, for example the incompressible Navier-Stokes equations
in the whole $\mathbb{R}^{3}$ space, which proposed by H. Navier and G. Stokes in 1845, to describe the evolution
of a viscous fluid:
$$(NS)\left\{\begin{array}{ll}
\partial_{t} u+u . \nabla u-\nu \Delta u & =-\nabla p \\ 
\operatorname{div} u & =0 \\
u_{t=0} & =u_{0}
\end{array}\right.$$
Here the parameter $ \nu> 0 $ is the viscosity of fluid and $ u (t, x) $ denotes the velocity field of a fluid at time $ t $ and at position $ x, $ subjected to a pressure $ p (t, x) $\\

 In $ 1934, $ J. Leray (see \cite{12}) gave (without proof) an blow-up result of the Navier-Stokes solution in Lebesgue spaces $ L ^ {p} \left(\mathbb {R} ^ {3} \right) $ for $ 3 <p \leq \infty $
$$
\| u (t) \| _ {L ^ {p}} \geq C \left (T ^ {*} - t \right) ^ {- \frac {p-3} {2 p}}
$$
which was proved by Y. Giga in 1986 (see \cite{10}). Moreover, several authors have been interested in this problem in Sobolev spaces $ \dot {H} ^ {s} \left (\mathbb {R} ^ {3} \right), s> 1/2 $
J. Benameur in 2010 (see \cite{J}), showed that for $ s> 5/2 $ the blow-up result index depends in an increasing way on the regularity index:
$$
\| u (t) \|_ {\dot {H}^{s}} \geq C \left (T ^ {*} - t \right) ^ {- s / 3}
$$
using $\operatorname{div} u=0,$ we obtain:
$$u . \nabla u=\operatorname{div}(u \otimes u)$$
thus the Navier-Stokes system is written in the form:
\begin{align}\label{ns}(NS)
\left\{\begin{array}{ll}
\partial_{t} u+\operatorname{div}(u \otimes u)-\nu \Delta u & =-\nabla p \\ 
\operatorname{div} u & =0 \\
u_{t=0} & =u_{0}
\end{array}\right.\end{align}
 and $u \otimes v=\left(u_{1} v,u_{2} v, u_{3} v\right) $, $u=\left(u_{1},u_{2}, u_{3}\right) \text { and } v=\left(v_{1},v_{2}, v_{3}\right)$.\\
 
The Weinstein operator $\Delta_{W}^{\alpha,d}$ has several applications in pure
and applied Mathematics especially in Fluid Mechanics ( \cite{6} )\\
The harmonic analysis associated with the Weinstein
operator is studied by Ben Nahia and Ben Salem (cf.
\cite{1}\cite{2}). In particular the authors have introduced and
studied the generalized Fourier transform associated with
the Weinstein operator. This transform is called the
Weinstein transform, and several authors have been interested in spaces related to this operator, in \cite{444} the authors introduced the Sobolev space associated with a Weinstein operator $H_{\mathcal{S}_{*}}^{s, \alpha}\left(\mathbb{R}_{+}^{d+1}\right)$ and investigated their propertie, and in \cite{mj1} the authors  introduced the Sobolev refine the inequality between the homogeneous
Weinstein-Besov spaces $\dot{\mathcal{B}}_{p, q}^{\mathrm{s}_{, j} \beta}\left(\mathbb{R}_{+}^{d+1}\right)$ and many more, such as the homogenous Weinstein-Riesz spaces $\mathcal{R}_{\beta}^{-s}\left(L_{\beta}^{p}\left(\mathbb{R}_{+}^{d+1}\right)\right)$ and the generalized Lorentz spaces.\\

This work is devoted to define and we study the Navier-Stokes equations associated with the Weinstein
operator (NSW) in the whole $\mathbb{R}_{+}^{d+1}$ space.\\

This paper is organized as follows : In section 2 we recall some elements
of harmonic analysis associated with the Weinstein operator, which will bee
needed in the sequel. In section 3 We study the existence and uniqueness of the solution the (NSW) equations  in $L_{\alpha}^{p}\left(\mathbb{R}_{+}^{d+1}\right), 2 \alpha+d+2<p \leq \infty$, and also we give some  blow-up results.
\section{\textbf{Harmonic analysis associated with the Weinstein-Laplace operator}}
\noindent\textbf{Notations.} In what follows, we need the following
notations:\newline$\bullet\;\mathbb{R}_{+}^{d+1}=\mathbb{R}^{d}\times\left]
0,\infty\right[  $.
\newline$\bullet\;x=(x_{1},...,x_{d},x_{d+1}%
)=(x^{\prime},x_{d+1})\in\mathbb{R}_{+}^{d+1}$
\newline$\bullet\; \vert x\vert=\sqrt{x_1^2+x_2^2+...+x_{d+1}^2}$.
\newline$\bullet\;\mathscr
C_{\ast}(\mathbb{R}^{d+1}),\;$the space of continuous functions on
$\mathbb{R}^{d+1}$, even with respect to the last variable.\newline%
$\bullet\;\mathscr C_{\ast,c}(\mathbb{R}^{d+1}),\;$the space of continuous
functions on $\mathbb{R}^{d+1}$ with compact support, even with respect to the
last variable.\newline$\bullet\;\mathscr C_{\ast}^{p}(\mathbb{R}^{d+1}),\;$the
space of functions of class $C^{p}$ on $\mathbb{R}^{d+1}$, even with respect
to the last variable.\newline$\bullet\;\mathscr E_{\ast}(\mathbb{R}^{d+1}%
),\;$the space of $C^{\infty}$-functions on $\mathbb{R}^{d+1}$, even with
respect to the last variable.\newline$\bullet\;\mathscr S_{\ast}%
(\mathbb{R}^{d+1}),\;$the Schwartz space of rapidly decreasing functions on
$\mathbb{R}^{d+1}$, even with respect to the last variable.\newline%
$\bullet\;\mathscr D_{\ast}(\mathbb{R}^{d+1}),\;$the space of $C^{\infty}%
$-functions on $\mathbb{R}^{d+1}$ which are of compact support, even with
respect to the last variable.\newline
\newline$\bullet\;L_{\alpha}%
^{p}(\mathbb{R}_{+}^{d+1}),$ $1\leq p\leq+\infty,\;$the space of measurable
functions on $\mathbb{R}_{+}^{d+1}$ such that
\[%
\begin{array}
[c]{lll}%
\Vert f\Vert_{\alpha,p} & = & \left[  \int_{\mathbb{R}_{+}^{d+1}}%
|f(x)|^{p}d\mu_{\alpha,d}(x)\right]  ^{\frac{1}{p}}<+\infty,\text{ if }1\leq
p<+\infty,\\
&  & \\
\Vert f\Vert_{\alpha,\infty} & = & \mathrm{ess}\underset{x\in\mathbb{R}%
	_{+}^{d+1}}{\sup}\left\vert f(x)\right\vert <+\infty,
\end{array}
\]
where $\mu_{\alpha,d}$ is the measure defined on $\mathbb{R}_{+}^{d+1}$ by
\begin{equation}
d\mu_{\alpha,d}(x)=\frac{x_{d+1}^{2\alpha+1}}{\left(  2\pi\right)  ^{\frac
		{d}{2}}2^{\alpha}\Gamma(\alpha+1)}dx, \label{1.2}%
\end{equation}
and $dx$ is the Lebesgue measure on $\mathbb{R}^{d+1}$. \newline%
$\bullet\;\mathcal{H}_{\ast}(\mathbb{C}^{d+1}\mathbb{)},\mathbb{\;}$the space
of entire functions on $\mathbb{C}^{d+1}$, even with respect to the last
variable, rapidly decreasing and of exponential type.

In this section, we shall collect some results and definitions from the theory
of the harmonic analysis associated with the Weinstein operator developed in
\cite{bn}.\newline

The Weinstein operator $\Delta_{W}^{\alpha,d}$ is defined on $\mathbb{R}%
_{+}^{d+1}=\mathbb{R}^{d}\times\left]  0,\ +\infty\right[  ,$ by:
\begin{equation}
\label{1.1}\Delta_{W}^{\alpha,d}=\sum_{i=1}^{d+1}\frac{\partial^{2}}{\partial
	x_{i}^{2} }+\frac{2\alpha+1}{x_{d+1}}\frac{\partial}{\partial x_{d+1}}%
=\Delta_{d}+L_{\alpha},~~\alpha>-\frac{1}{2},
\end{equation}
where $\Delta_{d}$ is the Laplacian for the $d$ first variables and
$L_{\alpha}$ is the Bessel operator for the last variable defined on $\left]
0,\ +\infty\right[  $ by :
\[
L_{\alpha}u=\frac{\partial^{2}u}{\partial x_{d+1}^{2}}+\frac{2\alpha
	+1}{x_{d+1}}\frac{\partial u}{\partial x_{d+1}}=\frac{1}{x_{d+1}^{2\alpha+1}
}\frac{\partial}{\partial x_{d+1}}\left[  x_{d+1}^{2\alpha+1}\frac{\partial
	u}{\partial x_{d+1}}\right]  .
\]
The Weinstein operator $\Delta_{W}^{\alpha,d}$, mostly referred to as the
Laplace-Bessel differential operator is now known as an important operator in
analysis. The relevant harmonic analysis associated with the Bessel
differential operator$L_{\alpha}$ goes back to S. Bochner, J. Delsarte, B.M.
Levitan and has been studied by many other authors such as J. L\"{o}fstr\"{o}m
and J. peetre \cite{Lof}, I. Kipriyanov \cite{Kip}, K.
Trim\`{e}che \cite{trim}, I.A. Aliev and B. Youssef \cite{Aliev2}.

Let us begin by the following result, which gives the eigenfunction
$\Psi_{\lambda}^{\alpha,d}$ of the Weinstein operator $\Delta_{W}^{\alpha,d}.$

\begin{propo}
	For all $\lambda=\left(  \lambda_{1},\lambda_{2},...,\lambda_{d+1}\right)
	\in\mathbb{C}^{d+1}$, the system
	\begin{equation}
	\left\{
	\begin{array}{lll}%
	\frac{\partial^{2}u}{\partial x_{j}^{2}}\left(  x\right)  =-\lambda_{j}%
	^{2}u(x),&\text{ if }~~~1\leq j\leq d\\
	L_{\alpha}u\left(  x\right)  =-\lambda_{d+1}^{2}u\left(  x\right)  ,\\
	u\left(  0\right)  =1,\;\frac{\partial u}{\partial x_{d+1}}(0)=0\\ \frac{\partial u}{\partial x_{j}}(0)=-i\lambda_{j},~~ &if ~~~1\leq j\leq
	d.
	\end{array}
	\right.  \label{2.1}%
	\end{equation}
	has a unique solution $\Psi^{\alpha}_{d} ({\lambda},.)\;$given by :
	\begin{equation}
	\forall z\mathbb{\in C}^{d+1},\;\Psi_{\lambda}^{\alpha,d}\left(  z\right)
	=e^{-i\left\langle z^{\prime}\text{,}\lambda^{\prime}\right\rangle }j_{\alpha
	}(\lambda_{d+1}z_{d+1}), \label{2.2}%
	\end{equation}
	where $z=(z^{\prime},x_{d+1}),\;z^{\prime}=\left(  z_{1},z_{2},...,z_{d}%
	\right)  $ and $j_{\alpha}$ is the normalized Bessel function of index
	$\alpha,\;$defined by
	$$
	\forall\xi\mathbb{\in C},\;j_{\alpha}(\xi)=\Gamma(\alpha+1)\underset{n=0}%
	{\sum^{\infty}}\frac{(-1)^{n}}{n!\Gamma(n+\alpha+1)}(\frac{\xi}{2})^{2n}.
	$$	
\end{propo}
\begin{propo}
	~~\newline i) For all $\lambda,\;z\in\mathbb{C}^{d+1}$ and $t\in\mathbb{R}$,
	we have
	\[
\Psi^{\alpha}_{d}\left(  \lambda,0\right)  =1,\;\Psi^{\alpha}_{d}\left(
	\lambda,z\right)  =\Psi^{\alpha}_{d}\left(  z,\lambda\right)  \;\text{and
	}\Psi^{\alpha}_{d}\left(  \lambda,tz\right)  =\Psi^{\alpha}_{d}\left(
	t\lambda,z\right)  .
	\]
	ii) For all $\nu\in\mathbb{N}^{d+1},\;x\in\mathbb{R}_{+}^{d+1}$ and
	$z\in\mathbb{C}^{d+1}$, we have
	\begin{equation}
	\label{2.3}|D_{z}^{\nu}\Psi^{\alpha}_{d}(x,z)|\leq\|x\|^{|\nu|}%
	\,\exp(\|x\|\,\|\operatorname{Im}z\|),
	\end{equation}
	where $D_{z}^{\nu}=\frac{\partial^{\nu}}{\partial z_{1}^{\nu_{1}}...\partial
		z_{d+1}^{\nu_{d+1}}}$ and $|\nu|=\nu_{1}+...+\nu_{d+1}.$ In particular
	\begin{equation}
	\label{2.4}\forall x,y\in\mathbb{R}_{+}^{d+1},\;|\Psi^{\alpha}_{d}%
	(x,y)|\leq1.
	\end{equation}
	
\end{propo}
\begin{defin}
	The Weinstein transform is given for $f\in L_{\alpha}^{1}(\mathbb{R}_{+}%
	^{d+1})$ by
	\begin{equation}
	\forall\lambda\in\mathbb{R}_{+}^{d+1},\;\mathscr F_{W}^{\alpha,d}%
	(f)(\lambda)=\int_{\mathbb{R}_{+}^{d+1}}f(x)\Psi^{\alpha}_{d}(x,\lambda
	)d\mu_{\alpha,d}(x). \label{2.10}%
	\end{equation}
	where $\mu_{\alpha,d}$ is the measure on $\mathbb{R}_{+}^{d+1}$ given by the
	relation (\ref{1.2}).
\end{defin}

The Weinstein tansform, referred to as the Fourier-Bessel transform, has been
investigated by I. Kipriyanov \cite{Kip}, I.A. Aliev \cite{Aliev1} and others
\newline( see \cite{Aliev3},\cite{bn},\cite{ bn1} and \cite{bet} ).

Using the properties of the classical Fourier transform on $\mathbb{R}^{d}$
and of the Bessel transform, one can easily see the following relation, which
will play an important role in the sequel.

\begin{example}
	Let $E_{s},\;s>0,$ be the function defined by
	\[
	\forall x\in\mathbb{R}^{d+1},\;E_{s}\left(  x\right)  =e^{-s\left\vert
		x\right\vert ^{2}}.
	\]
	Then the Weinstein transform $\mathscr F_{W}^{\alpha,d}$ of $E_{s}$ is given
	by :
	\begin{equation}
	\forall\lambda\in\mathbb{R}_{+}^{d+1},\;\mathscr F_{W}^{\alpha,d}%
	(E_{s})(\lambda)=\frac{1}{\left(  2s\right)  ^{\alpha+\frac{d}{2}+1}}%
	e^{-\frac{\left\vert \lambda\right\vert ^{2}}{4s}}. \label{2.11}%
	\end{equation}
	
\end{example}

\noindent Some basic properties of the transform $\mathscr F_{W}^{\alpha,d}$
are summarized in the following results. For the proofs, we refer to \cite{bn,
	bn1, bn2}.

\begin{propo}
	(see \cite{bn, bn1})~~\newline i) For all $f\in L_{\alpha}^{1}(\mathbb{R}%
	_{+}^{d+1})$, we have
	\begin{equation}
	\Vert\mathscr F_{W}^{\alpha,d}(f)\Vert_{\alpha,\infty}\leq\Vert f\Vert
	_{\alpha,1}.\label{2.12}%
	\end{equation}
	ii) For all $f\in L_{\alpha}^{1}(\mathbb{R}%
	_{+}^{d+1})$ and $\Delta_{W}^{\alpha,d}f\in L_{\alpha}^{1}(\mathbb{R}%
	_{+}^{d+1})$  we have
	\begin{align} \label{xx}
	\mathscr F_{W}^{\alpha,d}(\Delta_{W}^{\alpha,d}f)(x)=-\vert x\vert^2\mathscr F_{W}^{\alpha,d}(f)(x)
	\end{align}
\end{propo}
\begin{theorem}
	(see \cite{bn, bn1})~~\newline i\textbf{) }The Weinstein transform $\mathscr
	F_{W}^{\alpha,d}$ is a topological isomorphism \ from $\mathscr S_{\ast
	}(\mathbb{R}^{d+1})$ onto itself and from $\mathscr D_{\ast}(\mathbb{R}%
	^{d+1})$ onto $\mathcal{H}_{\ast}(\mathbb{C}^{d+1}\mathbb{)}$.\newline ii) Let
	$f\;\in\mathscr S_{\ast}(\mathbb{R}^{d+1})$. The inverse transform $\left(
	\mathscr F_{W}^{\alpha,d}\right)  ^{-1}\;$is given by
	\begin{equation}
	\forall x\in\mathbb{R}_{+}^{d+1},\;\left(  \mathscr F_{W}^{\alpha,d}\right)
	^{-1}(f)(x)=\mathscr F_{W}^{\alpha,d}(f)\left(  -x\right)  .\label{2.16}%
	\end{equation}
	iii) Let $f\in L_{\alpha}^{1}(\mathbb{R}_{+}^{d+1})$. If $\mathscr
	F_{W}^{\alpha,d}(f)\in L_{\alpha}^{1}(\mathbb{R}_{+}^{d+1}),$ then we have
	\begin{equation}
	f(x)=\int_{\mathbb{R}_{+}^{d+1}}\mathscr F_{W}^{\alpha,d}(f)\left(  y\right)
\Psi^{\alpha}_{d}(-x,y)d\mu_{\alpha,d}(y),\;a.e\;x\in\mathbb{R}_{+}%
	^{d+1}.\label{2.17}%
	\end{equation}
	
\end{theorem}

\begin{theorem}
	(see \cite{bn, bn1})~~\newline i) For all $f,g\in\mathscr S_{\ast}%
	(\mathbb{R}^{d+1}),$ we have the following Parseval formula
	\begin{equation}
	\int_{\mathbb{R}_{+}^{d+1}}f(x)\overline{g(x)}d\mu_{\alpha,d}(x)=\int
	_{\mathbb{R}_{+}^{d+1}}\mathscr F_{W}^{\alpha,d}(f)(\lambda)\overline{\mathscr
		F_{W}^{\alpha,d}(g)(\lambda)}d\mu_{\alpha,d}(\lambda).\label{2.18}%
	\end{equation}
	ii) ( Plancherel formula ). \newline For all $f\in\mathscr S_{\ast}%
	(\mathbb{R}^{d+1}),$ we have :
	\begin{equation}
	\int_{\mathbb{R}_{+}^{d+1}}\left\vert f(x)\right\vert ^{2}d\mu_{\alpha
		,d}(x)=\int_{\mathbb{R}_{+}^{d+1}}\left\vert \mathscr F_{W}^{\alpha
		,d}(f)(\lambda)\right\vert ^{2}d\mu_{\alpha,d}(\lambda).\label{2.19}%
	\end{equation}
	iii) ( Plancherel Theorem ) :\newline The transform $\mathscr F_{W}^{\alpha
		,d}$ extends uniquely to an isometric isomorphism on $L_{\alpha}%
	^{2}(\mathbb{R}_{+}^{d+1}).$
\end{theorem}

The following example is a consequence of the relation (\ref{2.11}).

\begin{example}
	Let $q_{t},\;t>0,$ be the function defined by
	\begin{equation}
	\forall x\in\mathbb{R}_{+}^{d+1},\;q_{t}\left(  x\right)  =\frac{1}{\left(
		2t\right)  ^{\alpha+\frac{d}{2}+1}}e^{-\frac{\left\Vert x\right\Vert ^{2}}%
		{4t}}. \label{2.20}%
	\end{equation}
	Then the inverse transform $\left(  \mathscr F_{W}^{\alpha,d}\right)  ^{-1}$
	of $q_{t}$ is given by
	\begin{equation}
	\forall x\in\mathbb{R}_{+}^{d+1},\;\left(  \mathscr F_{W}^{\alpha,d}\right)
	^{-1}\left( q_{t}\right)  (x)=\mathscr F_{W}^{\alpha,d}\left(  q
	_{t}\right)  \left(  x\right)  =e^{-t\left\Vert x\right\Vert ^{2}}.
	\label{2.21}%
	\end{equation}
	
\end{example}

\begin{defin}
	The translation operator $T_{x},\;$ $x\in\mathbb{R}_{+}^{d+1}$, associated
	with the Weinstein operator, is defined on $C_{\ast
	}(\mathbb{R}^{d+1}),$ for all $y\in\mathbb{R}_{+}^{d+1}$, by :
	\[
	T_{x}f\left(  y\right)  =\frac{a_{\alpha}}{2}\int_{0}^{\pi}f\left(  x^{\prime
	}+y^{\prime},\;\sqrt{x_{d+1}^{2}+y_{d+1}^{2}+2x_{d+1}y_{d+1}\cos\theta
	}\right)  \left(  \sin\theta\right)  ^{2\alpha}d\theta,
	\]
	where $x^{\prime}+y^{\prime}=\left(  x_{1}+y_{1},...,x_{d}+y_{d}\right)  $ and
	$a_{\alpha}$ is the constant given by (\ref{2.6}).
\end{defin}
The following propo summarizes some properties of the Weinstein
translation operator.

\begin{propo}
	(see \cite{bn, bn1})~~\newline i) For $f\in C_{\ast}(\mathbb{R}^{d+1})$, we
	have
	\[
	\forall x,\;y\in\mathbb{R}_{+}^{d+1},\;T_{x}f\left(  y\right)  =T_{y}f\left(
	x\right)  \text{ and }T_{0}f=f.
	\]
	ii) For all $f\in\mathscr E_{\ast}(\mathbb{R}^{d+1})$ and $y\in\mathbb{R}%
	_{+}^{d+1}$, the function $x\mapsto T_{x}f\left(  y\right)  $ belongs to
	$\mathscr E_{\ast}(\mathbb{R}^{d+1}).$\newline iii) We have
	\[
	\forall x\in\mathbb{R}_{+}^{d+1},\;\Delta_{W}^{\alpha,d}\circ T_{x}=T_{x}%
	\circ\Delta_{W}^{\alpha,d}.
	\]
	iv) Let $f\in L_{\alpha}^{p}(\mathbb{R}_{+}^{d+1}),\;1\leq p\leq+\infty$ and
	$x\in\mathbb{R}_{+}^{d+1}$. Then $T_{x}f$ belongs to $L_{\alpha}%
	^{p}(\mathbb{R}_{+}^{d+1})$ and we have
	\[
	\Vert T_{x}f\Vert_{\alpha,p}\leq\Vert f\Vert_{\alpha,p}.
	\]
	v) The function $\Psi^{\alpha}_{d}\left(  .,\lambda\right)  ,$ $\lambda
	\in\mathbb{C}^{d+1},\;$ satisfies on $\mathbb{R}_{+}^{d+1}$ the following
	product formula:
	\begin{equation}
	\forall y\in\mathbb{R}_{+}^{d+1},\;\Psi^{\alpha}_{d}\left(  x,\lambda\right)
\Psi^{\alpha}_{d}\left(  y,\lambda\right)  =T_{x}\left[  \Lambda_{\alpha
		,d}\left(  .,\lambda\right)  \right]  \left(  y\right)  .\label{2.23}%
	\end{equation}
	\newline vi) Let $f\in L_{\alpha}^{p}(\mathbb{R}_{+}^{d+1}),\;p=1$ or $2$ and
	$x\in\mathbb{R}_{+}^{d+1}$, we have
	\begin{equation}
	\forall y\in\mathbb{R}_{+}^{d+1},\;\mathscr F_{W}^{\alpha,d}\left(
	T_{x}f\right)  \left(  y\right)  =\Psi^{\alpha}_{d}\left(  x,y\right)
	\mathscr F_{W}^{\alpha,d}\left(  f\right)  \left(  y\right)  .\label{2.24}%
	\end{equation}
	vii) The space $\mathscr S_{\ast}(\mathbb{R}^{d+1})$ is invariant under the
	operators $T_{x},\;x\in\mathbb{R}_{+}^{d+1}.$\newline
\end{propo}

\begin{defin}
	The Weinstein convolution product of $f,g\in\mathscr C_{\ast}(\mathbb{R}%
	^{d+1})$ is given by:
	\begin{equation}
	\forall x\in\mathbb{R}_{+}^{d+1},\;f\ast_{W}g\left(  x\right)  =\int
	_{\mathbb{R}_{+}^{d+1}}T_{x}f\left(  y\right)  g\left(  y\right)  d\mu
	_{\alpha,d}(y).\label{2.25}%
	\end{equation}
\end{defin}
  
\begin{propo}
	(see \cite{bn, bn1})~~\newline i) Let $p,q,r\in\left[  1,\;+\infty\right]  $
	such that $\frac{1}{p}+\frac{1}{q}-\frac{1}{r}=1.$ Then for all $f\in
	L_{\alpha}^{p}(\mathbb{R}_{+}^{d+1})$ and$\;g\in L_{\alpha}^{q}(\mathbb{R}%
	_{+}^{d+1}),$ the function $f\ast_{W}g$ $\in L_{\alpha}^{r}(\mathbb{R}%
	_{+}^{d+1})$ and we have
	\begin{equation}
	\Vert f\ast_{W}g\Vert_{\alpha,r}\leq\Vert f\Vert_{\alpha,p}\Vert
	g\Vert_{\alpha,q}.\label{2.26}%
	\end{equation}
	ii) For all $f,g\in L_{\alpha}^{1}(\mathbb{R}_{+}^{d+1}),\;\left(
	resp.\;\mathscr S_{\ast}(\mathbb{R}^{d+1})\right)  ,\;f\ast_{W}g$ $\in
	L_{\alpha}^{1}(\mathbb{R}_{+}^{d+1})$ $\left(  resp.\;\mathscr S_{\ast
	}(\mathbb{R}^{d+1})\right)  \;$and we have
	\begin{equation}
	\mathscr F_{W}^{\alpha,d}(f\ast_{W}g)=\mathscr F_{W}^{\alpha,d}(f)\mathscr
	F_{W}^{\alpha,d}(g).\label{2.27}%
	\end{equation}
	
\end{propo}

\section{\textbf{Navier-Stokes equation associated with the Weinstein operator}}
 In this section, we collect some notations and definitions that will be used later\\
$\bullet$ $\Delta_{W}^{\alpha,d} f:=\left(\Delta_{W}^{\alpha,d} f_{1}, \Delta_{W}^{\alpha,d}f_{2}, . ., \Delta_{W}^{\alpha,d} f_{d+1}\right)$: Weinstein-Laplace of $f$ if $ f=\left(f_{1}, f_{2}, . ., f_{d+1}\right)$\\
$\bullet$ $\mathscr F_{W}^{\alpha,d}(f):=\left(\mathscr F_{W}^{\alpha,d} f_{1}, \mathscr F_{W}^{\alpha,d}f_{2}, . ., \mathscr F_{W}^{\alpha,d} f_{d+1}\right)$: Weinstein transform of $f$.\\
 $\bullet$ $\left(L_{\alpha}^{p}\left(\mathbb{R}_{+}^{d+1}\right)\right)^{d+1}=L_{\alpha}^{p}\left(\mathbb{R}_{+}^{d+1}\right)\times...\times L_{\alpha}^{p}\left(\mathbb{R}_{+}^{d+1}\right)$\\
$\bullet$ $C\left([0, T^{*}[,\left(\left(L_{\alpha}^{p}\left(\mathbb{R}_{+}^{d+1}\right)\right)^{d+1}\right)\right):$ Space of continuous functions of $[0, T^{*}[$ in$ \left(L_{\alpha}^{p}\left(\mathbb{R}_{+}^{d+1}\right)\right)^{d+1}.$\\
$\bullet$ $u_j \otimes v=\left(u_{j} v_{1},..., u_jv_{d+1} \right), u=\left(u_{1}, ..., u_{d+1}\right) \text { and } v=\left(v_{1},..., v_{d+1}\right)$\\
$\bullet$ $
u \otimes v=\left(u_1 \otimes v,..., u_{d+1} \otimes v\right).$
\begin{defins}
	Let $ f $ is a function verified, $f\in L_{\alpha}%
	^{1}(\mathbb{R}_{+}^{d+1})$ and  $\lambda_{}\longmapsto ~\lambda_{j} \mathscr F_{W}^{\alpha,d}(f)(\lambda)\in L_{\alpha}
	^{1}(\mathbb{R}_{+}^{d+1})$ , $j=1,2,...,d+1 $. The Weinstein gradient $\nabla_{W}^{\alpha, d}(f) (x)$ in $ x = $ $ \left (x_ {1}, \ldots, x_ {d+1} \right) \in \mathbb {R}^ {d+1}_{+}, $ defined by
\begin{align}
\nabla_{W}^{\alpha, d}f(x) :=&i\left(\mathscr F_{W}^{\alpha,d}\right)^{-1}\left(~ \lambda .\mathscr F_{W}^{\alpha,d}f(\lambda)\right)(x)\\=&i\left(\begin{array}{l}
\left(\mathscr F_{W}^{\alpha,d}\right)^{-1}\left( \lambda_{1} \mathscr F_{W}^{\alpha,d}f(\lambda)\right)(x)   \\
\left(\mathscr F_{W}^{\alpha,d}\right)^{-1}\left( \lambda_{2} \mathscr F_{W}^{\alpha,d}f(\lambda)\right)(x)   \\
\quad\quad\quad\quad\quad \vdots  \\
\left(\mathscr F_{W}^{\alpha,d}\right)^{-1}\left( \lambda_{d+1} \mathscr F_{W}^{\alpha,d}f(\lambda)\right)(x) 
\end{array}\right)
\end{align}
	Let $ f=(f_1,f_1,. .,f_{d+1})$ is a function verified, $f\in \left(L_{\alpha}^{1}(\mathbb{R}_{+}^{d+1})\right)^{d+1}$ and  $\lambda_{}\longmapsto\lambda_{j} \mathscr F_{W}^{\alpha,d}(f_{j})(\lambda)\in L_{\alpha}^{1}(\mathbb{R}_{+}^{d+1})$, $j=1,2,...,d+1 $. The Weinstein divergence $\operatorname{div}_{W}^{\alpha, d}  (f)$ defined by
\begin{align}
\operatorname{div}_{W}^{\alpha, d}  (f)(x) :=i\sum_{j=1}^{d+1}\left(\mathscr F_{W}^{\alpha,d}\right)^{-1}\left(\lambda_{j} \mathscr F_{W}^{\alpha,d}f_{j}(\lambda)\right)(x)
\end{align}
\end{defins}
\begin{remarks}
	i) By the relation (\ref{xx}) we have
$$\operatorname{div}_{W}^{\alpha, d}\nabla_{W}^{\alpha, d}=\Delta_{W}^{\alpha, d}$$
 	ii) If $\alpha=-\frac{1}{2}$ we have
 	\begin{align*}
 	\operatorname{div}_{W}^{\alpha, d} &=\operatorname{div} \\
 	\nabla_{W}^{\alpha, d}&=\nabla
 	\end{align*}
 	iii)
 \begin{align*}
 \operatorname{div}_{\mathrm{W}}^{\alpha, d}(u \otimes v)&=\Big(\operatorname{div}_{\mathrm{W}}^{\alpha, d}(u_{1}\otimes v), ...,\operatorname{div}_{\mathrm{W}}^{\alpha, d}(u_{d+1}\otimes v)\Big),\\
  &=i\sum_{k=1}^{d+1}\left(\mathscr F_{W}^{\alpha,d}\right)^{-1}\left( \lambda_{k}\mathscr F_{W}^{\alpha,d}(v_{k}\otimes u)(\lambda)\right)
 \end{align*}
 so
  \begin{align*}
\mathscr F_{W}^{\alpha,d}\left( \operatorname{div}_{\mathrm{W}}^{\alpha, d}(u \otimes v)\right)(\lambda)=i\sum_{k=1}^{d+1}\lambda_{k}\mathscr F_{W}^{\alpha,d}(v_{k}\otimes u)(\lambda)
 \end{align*} 
\end{remarks}
We consider in the rest of this article that the incompressible Navier-
Stokes-Weinstein system is given by:
\begin{equation}\label{11}\left( NSW\right) \left\{\begin{array}{lll}
\partial_{t} u-\nu \Delta_{W}^{\alpha, d} u+\operatorname{div}_{\mathrm{W}}^{\alpha, d}(u \otimes u)=-\nabla_{W}^{\alpha, d} p, &\quad \text { in } \quad \mathbb{R}_{+}^{*} \times \mathbb{R}_{+}^{d+1} \\
\operatorname{div}_{\mathrm{W}}^{\alpha, d} u=0 &\quad \text { in } \quad \mathbb{R}_{+}^{*} \times \mathbb{R}_{+}^{d+1}\\
u(0)=u^{0} &\quad \text { in } \quad  \mathbb{R}_{+}^{d+1}
\end{array}\right.\end{equation}
In the case $\alpha=-\frac{1}{2}$, Navier-
Stokes-Weinstein system $(NSW)$ reduces to the  classical Navier-Stokes system $(NS)$ $(see \eqref{ns})$.\\
with:\\
$\bullet ~\nu$ is the fluid viscosity.\\ 
$\bullet$ $u=u(t, x)=\left(u_{1},..., u_{d+1}\right): \mathbb{R}_{+} \times \mathbb{R}^{d+1} \rightarrow \mathbb{R}^{d+1}$ is the fluid velocity field,\\ 
 $\bullet$ $p=p(t, x): \mathbb{R}_{+} \times \mathbb{R}^{d+1} \rightarrow \mathbb{R}$ is the fluid pressure.
\\ $\bullet$ $u$ and $p$ are the two unknowns of the system.\\
\begin{remark}
 If $u^{0}$ is regular, applying the divergence operator to the Navier-Stokes-Weinstein equation, we can express the $p$ pressure as a function of the fluid velocity $u$
\end{remark}
\[
\operatorname{div}_{\mathrm{W}}^{\alpha, d}\left(\operatorname{div}_{\mathrm{W}}^{\alpha, d}(u \otimes u)\right)=-\Delta_{W}^{\alpha, d} p
\]
So
\[
p=\left(-\Delta_{W}^{\alpha, d}\right)^{-1}\Big(-\sum_{k, j=1}^{d+1} \left(\mathscr F_{W}^{\alpha,d}\right)^{-1}\left(\lambda_{k}\lambda_{j} \mathscr F_{W}^{\alpha,d}(u_{j}u_{k})(\lambda)\right)\Big)
\]
Moreover, under the same conditions, Duhamel's formula implies
$$u(t, x)=e^{\nu t \Delta_{W}^{\alpha, d}} u^{0}-\int_{0}^{t} e^{\nu(t-s) \Delta_{W}^{\alpha, d}}\left(\operatorname{div}_{\mathrm{W}}^{\alpha, d}(u \otimes u)+\nabla_{W}^{\alpha, d} p\right) d s$$
then 
$$u(t, x)=e^{\nu t \Delta_{W}^{\alpha, d}} u^{0}-\int_{0}^{t} e^{\nu(t-s) \Delta_{W}^{\alpha, d}}\mathbb{P}\left(\operatorname{div}_{\mathrm{W}}^{\alpha, d}(u \otimes u)\right) d s$$
with:\\
$\bullet$ $e^{\nu t \Delta_{W}^{\alpha, d}} u=q_{\nu t} *_{W} u=\frac{1}{(\left(
	2\nu t\right)  ^{\alpha+\frac{d}{2}+1}} e^{\frac{-|x|^{2}}{4 \nu t}} *_{W} u=\left(\mathcal{F}_{W}^{\alpha, d}\right)^{-1}\left(e^{-\nu t|\xi|^{2}} \mathcal{F}_{W}^{\alpha, d}{u}\right)$\\
$\bullet$ $\mathbb{P}$ designates the Leray projector defined by:
$$\mathbb{P}=I-\nabla_{W}^{\alpha, d}\left(\Delta_{W}^{\alpha, d}\right)^{-1} \operatorname{div}_{W}^{\alpha, d}$$
The operator $\mathcal{F}_{W}^{\alpha, d}({\mathbb{P}})$ is a matrix and the coefficient $\mathcal{F}_{W}^{\alpha, d}({\mathbb{P}}_{i, j})$ defined as follows:
$$
\mathcal{F}_{W}^{\alpha, d}({\mathbb{P}}_{i, j})=\delta_{i, j}-\frac{\xi_{j} \xi_{k}}{|\xi|^{2}}=\left\{\begin{array}{ll}
1-\frac{\xi_{j}^2}{|\xi|^{2}} ~~&\text{if } i=j \\ 
-\frac{\xi_{j} \xi_{k}}{|\xi|^{2}} &\text{if not }
\end{array}\right.
$$
\subsection {Existence and uniqueness of solution for the
	(NSW) system in  $L_{\alpha}^{p}\left(\mathbb{R}_{+}^{d+1}\right)$ spaces}
\begin{theorem}	\label{th}
 Let $2 \alpha+d+2<p \leq \infty,$ fixed and $u^{0} \in\left(L_{\alpha}^{p}\left(\mathbb{R}_{+}^{d+1}\right)\right)^{d+1}$ such that
$\operatorname{div}_{W}^{\alpha, d}(u)=0 .$ Then they exist $T^{*}>0$ and a unique solution
$u$ for the Navier-Stokes-Weinstein system in $C\left(\left[0, T^{*}\left[,\left(L_{\alpha}^{p}\left(\mathbb{R}_{+}^{d+1}\right)\right)^{d+1}\right)\right.\right.$\\
\end{theorem}
\begin{proof}
	For $R>0$ and $T>0$ we pose
	$$B_{R, T}=\left\{u \in C\left([0, T],\left(L_{\alpha}^{p}\left(\mathbb{R}_{+}^{d+1}\right)\right)^{d+1}\right) /\|u\|_{L^{\infty}\left([0, T], L_{\alpha}^{p}\right)} \leq R\right\}$$
	We consider the following application
	$$\begin{array}{l}
	\varphi: B_{R, T} \rightarrow C\left([0, T],\left(L_{\alpha}^{p}\left(\mathbb{R}_{+}^{d+1}\right)\right)^{d+1}\right) \\
	\quad u \mapsto e^{\nu t \Delta_{W}^{\alpha, d}} u^{0}-\int_{0}^{t} e^{\nu(t-\sigma) \Delta_{W}^{\alpha, d}} \mathbb{P}\left(\operatorname{div}_{W}^{\alpha, d}(u \otimes v)\right) d \sigma
	\end{array}$$
	we pose
	\begin{equation}\left\{\begin{array}{l}
	L_{0}(t)=e^{\nu t \Delta_{W}^{\alpha, d}} u^{0} \\
	B(u, v)(t)=\int_{0}^{t} e^{\nu(t-\sigma) \Delta_{W}^{\alpha, d}} \mathbb{P}\left(\operatorname{div}_{\mathrm{W}}^{\alpha, d}(u \otimes v)\right) d \sigma
	\end{array}\right.
	\end{equation}
	We want to apply Fixed Point Theorem, so we look for a good choice
	of $R$ and $T$
	we're going to do some research for the first condition on $R$ and $T$ such that $\varphi\left(B_{R, T}\right) \subset B_{R, T} .$ we have,
	\begin{equation}\left\|L_{0}(t)\right\|_{L_{\alpha}^{p}}=\left\|e^{\nu t \Delta_{W}^{\alpha, d}} u^{0}\right\|_{L_{\alpha}^{p}}=\left\|q_{\nu t} *_{W} u^{0}\right\|_{L_{\alpha}^{p}} \leq\left\|q_{\nu t}\right\|_{L_{\alpha}^{1}}\left\|u^{0}\right\|_{L_{\alpha}^{p}} \leq\left\|u^{0}\right\|_{L_{\alpha}^{p}}\end{equation}
Furthermore,
$$\begin{aligned}
\|B(u, v)(t)\|_{L_{\alpha}^{p}} &=\left\|\int_{0}^{t} e^{\nu(t-\sigma) \Delta_{W}^{\alpha, d}} \mathbb{P}\left(\operatorname{div}_{\mathrm{W}}^{\alpha, d}(u \otimes v)\right) d \sigma\right\|_{L_{\alpha}^{p}} \\
& \leq \int_{0}^{t}\left\|e^{\nu(t-\sigma) \Delta_{W}^{\alpha, d}} \mathbb{P}\left(\operatorname{div}_{\mathrm{W}}^{\alpha, d}(u \otimes v)\right)\right\|_{L_{\alpha}^{p}} d \sigma
\end{aligned}$$
Using the inequality of Young with
$$ 1+\frac{1}{p}=\frac{2}{p}+\frac{p-1}{p}$$
we obtain,
\begin{equation}\begin{aligned}
\left\|e^{\nu(t-\sigma) \Delta_{W}^{\alpha, d}} \mathbb{P}\left(\operatorname{div}_{\mathrm{W}}^{\alpha, d}(u \otimes v)\right)\right\|_{L_{\alpha}^{p}} & \leq \sum_{j=1}^{d+1}\left\|\left(\mathcal{F}_{W}^{\alpha, d}\right)^{-1}\left(e^{-\nu(t-\sigma)|\xi|^{2}} M(\xi) \mathrm{i} \xi_{j} \mathcal{F}_{W}^{\alpha, d}\left(u_{j} \otimes v\right)\right)\right\|_{L_{\alpha}^{p}} \\
& \leq \sum_{j=1}^{d+1}\left\|\left(\mathcal{F}_{W}^{\alpha, d}\right)^{-1}\left(\mathrm{i} \xi_{j} e^{-\nu(t-\sigma)|\xi|^{2}} M(\xi)\right) *_{W}\left(u_{j} \otimes v\right)\right\|_{L_{\alpha}^{p}} \\
& \leq\left\|\left(\mathcal{F}_{W}^{\alpha, d}\right)^{-1}\left(\mathrm{i} \xi_{j} e^{-\nu(t-\sigma)|\xi|^{2}} M(\xi)\right)\right\|_{L_{\alpha}^{\frac{p}{p-1}}}\|u \otimes v\|_{L_{\alpha}^{\frac{p}{2}}}
\end{aligned}\end{equation}
The inequality of Young implies
\begin{equation}\begin{aligned}
\left\|e^{\nu(t-\sigma) \Delta_{W}^{\alpha, d}} \mathbb{P}(\operatorname{div}_{\mathrm{W}}^{\alpha, d}(u \otimes v))\right\|_{L_{\alpha}^{p}} & \leq\left\|\left(\mathcal{F}_{W}^{\alpha, d}\right)^{-1}\left(\mathrm{i} \xi_{j} e^{-\nu(t-\sigma)|\xi|^{2}} M(\xi)\right)\right\|_{L_{\alpha}^{\frac{p}{p-1}}}\|u\|_{L_{\alpha}^{p}}\|v\|_{L_{\alpha}^{p}} \\
& \leq R^{2}\left\|\left(\mathcal{F}_{W}^{\alpha, d}\right)^{-1}\left(\mathrm{i} \xi_{j} e^{-\nu(t-\sigma)|\xi|^{2}} M(\xi)\right)\right\|_{L_{\alpha}^{\frac{p}{p-1}}}
\end{aligned}\end{equation}
with $M(\xi)=\frac{1}{|\xi|^{2}}\left(\delta_{i j}|\xi|^{2}-\xi_{i} \xi_{j}\right)_{1 \leq i, j \leq d+1}$
we have,\\
\begin{equation*}\begin{aligned}
&\left\|\left(\mathcal{F}_{W}^{\alpha, d}\right)^{-1}\left(e^{-\nu(t-\sigma)|\xi|^{2}} \mathrm{i} \xi_{j} \frac{\xi_{i} \xi_{j}}{|\xi|^{2}}\right)\right\|_{L_{\alpha}^{p-1}}\\&
\leq\left\|\int \Psi^{\alpha}_{d}(\xi,-x) e^{-\nu(t-\sigma)|\xi|^{2}} \mathrm{i} \xi_{j} \frac{\xi_{i} \xi_{j}}{|\xi|^{2}} d\mu_{\alpha,d} (\xi)\right\|_{L_{\alpha}^{\frac{p}{p-1}}}\\&
\leq\left\|(\nu(t-\sigma))^{-\frac{2\alpha+d+3}{2}} \int \Psi^{\alpha}_{d}\left(-(\nu(t-\sigma))^{-1 / 2} x, \eta\right) e^{-|\eta|^{2}} \mathrm{i} \eta_{j} \frac{\eta_{i} \eta_{j}}{|\eta|^{2}} d\mu_{\alpha,d} (\eta)\right\|_{L^{\frac{p}{p-1}}_{\alpha}}\\&
\leq\left\|(\nu(t-\sigma))^{-\frac{2\alpha+d+3}{2}}\left(\mathcal{F}_{W}^{\alpha, 3}\right)^{-1}\left(e^{-|\eta|^{2}} \mathrm{i} \eta_{j} \frac{\eta_{i} \eta_{j}}{|\eta|^{2}}\right)\left((\nu(t-\sigma))^{-1 / 2} .\right)\right\|_{L_{\alpha}^{\frac{p}{p-1}}}\\&
\leq(\nu(t-\sigma))^{-\frac{p+2 \alpha+d+2}{2 p}}\left\|\left(\mathcal{F}_{W}^{\alpha, 3}\right)^{-1}\left(e^{-|\eta|^{2}} \mathrm{i} \eta_{j} \frac{\eta_{i} \eta_{j}}{|\eta|^{2}}\right)\right\|_{L_{\alpha}^{\frac{p}{p-1}}}
\end{aligned}\end{equation*}
Similarly, we have
\begin{align*}
&\left\|\left(\mathcal{F}_{W}^{\alpha, d}\right)^{-1}\left(\mathrm{i} \xi_{j} e^{-\nu(t-\sigma)|\xi|^{2}}\left(1-\frac{\xi_{i}^{2}}{|\xi|^{2}}\right)\right)\right\|_{L_{\alpha}^{\frac{p}{p-1}}} \\& \leq(\nu(t-\sigma))^{-\frac{p+2 \alpha+d+2}{2 p}}\left\|\left(\mathcal{F}_{W}^{\alpha, d}\right)^{-1}\left(\mathrm{i} \eta_{j} e^{-|\eta|^{2}}\left(1-\frac{\eta_{i}^{2}}{|\eta|^{2}}\right) \right\|_{L_{\alpha}^{\frac{p}{p-1}}}\right.
\end{align*}
Let $f, g$ such that
\[
f(\eta)=\mathrm{i} \eta_{j} e^{-|\eta|^{2}}\left(1-\frac{\eta_{i}^{2}}{|\eta|^{2}}\right) \quad \text { and } \quad g(\eta)=e^{-|\eta|^{2}} \mathrm{i} \eta_{i} \frac{\eta_{i} \eta_{j}}{|\eta|^{2}}
\]
Then
\[
\begin{aligned}
\|B(u, v)\|_{L_{\alpha}^{p}} & \leq R^{2} \int_{0}^{T}(\nu(T-\sigma))^{-\frac{p+2 \alpha+d+2}{2 p}}\left(\left\|\left(\mathcal{F}_{W}^{\alpha, d}\right)^{-1}(f)\right\|_{L_{\alpha}^{{\frac{p}{p-1}}}}+\left\|\left(\mathcal{F}_{W}^{\alpha, d}\right)^{-1}(g)\right\|_{L_{\alpha}^{\frac{p}{p-1}}}\right) d \sigma \\
& \leq R^{2} \frac{\nu^{-\frac{p+2 \alpha+d+2}{2 p}} T^{\frac{p-2 \alpha-d-2}{2 p}}}{\frac{p-2 \alpha-d-2}{2 p}}\left(\left\|\left(\mathcal{F}_{W}^{\alpha, d}\right)^{-1}(f)\right\|_{L_{\alpha}^{\frac{p}{p-1}}}+\left\|\left(\mathcal{F}_{W}^{\alpha, d}\right)^{-1}(g)\right\|_{L_{\alpha}^{\frac{p}{p-1}}}\right) \\
& \leq C R^{2} T^{\frac{p-2 \alpha-d-2}{2 p}}
\end{aligned}
\]
with $C=\frac{\nu^{-\frac{p+2 \alpha+d+2}{2 p}}}{\frac{p-2 \alpha-d-2}{2 p}}$\\
then we choose $R>\left\|u^{0}\right\|_{L_{\alpha}^{p}}$ and $T>0$ such as
\begin{align}\label{tt }
\|\varphi(u)\|_{L_{\alpha}^{p}} \leq\left\|u^{0}\right\|_{L_{\alpha}^{p}}+C R^{2} T^{\frac{p-2 \alpha-d-2}{2 p}} \leq R
\end{align}
From where $, \varphi\left(B_{R, T}\right) \subset B_{R, T}$\\
 we're going to do some research for the second condition on $R$ and $T$ where $\varphi$ is contracting:
we have
\begin{equation*}\begin{aligned}
\|\varphi(u)-\varphi(v)\|_{L_{\alpha}^{p}} & \leq\|B(u, u)(t)-B(v, v)(t)\|_{L_{\alpha}^{p}} \\
& \leq\|B(u-v, u)(t)+B(v, u-v)(t)\|_{L_{\alpha}^{p}} \\
& \leq C T^{\frac{p-2 \alpha-d-2}{2 p}}\left(\sup _{0 \leq t \leq T}\|u(t, x)\|_{L_{\alpha}^{p}}+\sup _{0 \leq t \leq T}\|v(t, x)\|_{L_{\alpha}^{p}}\right) \\
& \times \sup _{0 \leq t \leq T}\|(u-v)(t, x)\|_{L_{\alpha}^{p}} \\
& \leq 2 C R T^{\frac{p-2 \alpha-d-2}{2 p}} \sup _{0 \leq t \leq T}\|(u-v)(t, x)\|_{L_{\alpha}^{p}}
\end{aligned}\end{equation*}
We choose $T$ and $R$ as
\begin{equation}\label{t}
2 C R T^{\frac{p-2 \alpha-d-2}{2 p}}<1 / 2\end{equation}
So, according to the inequalities \eqref{tt } and \eqref{t}, the theorem of the fixed point implies the existence of a unique solution of the system $(N S W)$ $ u$ in $C\left([0, T],\left(L_{\alpha}^{p}\left(\mathbb{R}^{d+1}\right)^{d+1}\right)\right.$
\end{proof}
\begin{remark}
 We can choose $T = \frac{C_{p,\alpha,d}}{\left\| u^{0} \right\|_{L ^ {p}_{\alpha}}^{\frac{2 p}{p-2 \alpha-d-2}}}: $ depends only on the norm of $ u^{0} $\\
 Indeed, we take $R = 2 \left\| u^{0} \right\|_{L^ {p}_{\alpha}} $ and using the inequalities $ \eqref{tt } and \eqref{t} , $ we can deduct
\[
T^{\frac{p-2 \alpha-d-2}{2 p}}=\frac{C_{0}}{\left\|u^{0}\right\|_{L^{p}_{\alpha}}}
\]
then
\[
T = \frac{C_{p,\alpha,d}}{\left\| u^{0} \right\|_{L ^ {p}_{\alpha}}^{\frac{2 p}{p-2 \alpha-d-2}}}
\]
 with $C_{p,\alpha,d}=C_{0}^{\frac{2 p}{p-2 \alpha-d-2}}$.
\end{remark}
\begin{propo}
Let $ \nu> 0,  ~ 2\alpha+d+2 <p \leq \infty $ and $ u \in C \left([0, T ^ {*} [; \left(L ^ {p}_{\alpha} \left(\mathbb{R} ^ {d+1} \right) \right) ^ {d+1} \right)$ the solution of (NSW) such as  $u \notin C\left(\left[0, T ^ {* } \right]; \left(\left(L ^ {p}_{\alpha}  \left(\mathbb{R} ^ {d+1} \right) \right) ^ {d+1} \right)\right) $ with $ T ^ {*} <\infty. $ Then
$$\limsup _{t \nearrow T^{*}}\|u(t)\|_{L^{p}_{\alpha} }=+\infty$$

\end{propo}
Let $ T_{1} = T = \frac{C_{p,\alpha,d}}{\left\| u^{0} \right\|_{L ^ {p}_{\alpha}}^{\frac{2 p}{p-2 \alpha-d-2}}}, $ then according to Theorem \ref{th} , there is a unique solution
 $ u \in C \left(\left[0, T_ {1} \right], \left(L^{p }_{\alpha} \left(\mathbb{R}^{d+1} \right) \right)^{d+1} \right)  $ of $(NSW)$.\\
 We consider the following system:
 \begin{equation}\left( NSW_1\right)\left\{\begin{array}{l}
 \partial_{t} v-\nu \Delta_{W}^{\alpha, d} v+\operatorname{div}_{\mathrm{W}}^{\alpha, d}(v \otimes v)=-\nabla_{W}^{\alpha, d} p,  \\
 \operatorname{div}_{{W}}^{\alpha, d} v=0 \\
 v(0)=u(T_1)
 \end{array}\right.\end{equation}
So, there is a unique solution for $(NSW_1)$ $v\in C \left([0, T_  {1} ]; \left(L ^ {p}_{\alpha} \left(\mathbb{R} ^ {d+1} \right) \right) ^ {d+1} \right)$

$$T_{2}=\frac{C_{p,\alpha,d}}{\left\| u(0) \right\|_{L ^ {p}_{\alpha}}^{\frac{2 p}{p-2 \alpha-d-2}}}=\frac{C_{p,\alpha,d}}{\left\| u({T_{1}}) \right\|_{L ^ {p}_{\alpha}}^{\frac{2 p}{p-2 \alpha-d-2}}}$$
So, by uniqueness of solution, we have
$$u(t)=v\left(t-T_{1}\right), \quad \forall t\in[T_  {1}, T_  {2} ]$$

moreover, $ u \in C \left (\left [0, T_ {1} + T_ {2} \right], \left (L^{p}_{\alpha} \left (\mathbb {R}^{d+1} \right) \right)^{d+1} \right) $ then $ T_{1} + T_{2} <T ^ {*} $\\
We can then construct a series $ T_ {1}, T_ {2}, \ldots, T_ {n} $ with
$$
T_{k} = \frac {C_ {p,\alpha,d}} {\left\| u \left(T_ {1} + T_ {2} + \ldots + T_ {k-1} \right) \right\|_{ L^ {p}_{\alpha}}^{\frac{2 p}{p-2 \alpha-d-2}}} \quad\forall 2 \leq k \leq n
$$
We now consider the following system:

 \begin{equation}\left( NSW_n\right)\left\{\begin{array}{l}
\partial_{t} w-\nu \Delta_{W}^{\alpha, d} w+\operatorname{div}_{\mathrm{W}}^{\alpha, d}(w \otimes w)=-\nabla_{W}^{\alpha, d} p,  \\
\operatorname{div}_{{W}}^{\alpha, d} w=0 \\
w(0)=u(T_ {1} + T_ {2} + \ldots + T_ {n})
\end{array}\right.\end{equation}
By uniqueness of the solution,
$$u(t)=w\left(t-\left(T_{1}+T_{2}+. .+T_{n}\right)\right), \quad \forall T_{1}+T_{2}+\ldots+T_{n} \leq t \leq T_{1}+T_{2}+\ldots+T_{n+1}$$
So
$$T_{1}+T_{2}+\ldots+T_{n+1}<T^{*}$$
then $ \sum_{k=1}^{n} T_{k}$ converges, which implies
$$\sum_{n=1}^{\infty} \frac{C_{ {p,\alpha,d}}}{\left\|u\left(\sum_{k=1}^{n} T_{k}\right)\right\|_{ L^ {p}_{\alpha}}^{\frac{2 p}{p-2 \alpha-d-2}}}<\infty$$
from where,
$$\lim_{m \rightarrow \infty}\sum_{n=m+1}^{\infty} \frac{C_{ {p,\alpha,d}}}{\left\|u\left(\sum_{k=1}^{n} T_{k}\right)\right\|_{ L^ {p}_{\alpha}}^{\frac{2 p}{p-2 \alpha-d-2}}}=0 $$
so
$$\lim_{n \rightarrow \infty} \frac{C_{ {p,\alpha,d}}}{\left\|u\left(\sum_{k=1}^{n} T_{k}\right)\right\|_{ L^ {p}_{\alpha}}^{\frac{2 p}{p-2 \alpha-d-2}}}=0 $$
So he comes,
$$\lim_{n \rightarrow \infty} {\left\|u\left(\sum_{k=1}^{n} T_{k}\right)\right\|_{ L^ {p}_{\alpha}}^{\frac{2 p}{p-2 \alpha-d-2}}}=+\infty $$
As a result, we can deduce
$$\sum_{k=1}^{\infty} T_{k}=T^{*} \operatorname{and~~} \limsup _{t \nearrow T^{*}}{\left\|u\left(\sum_{k=1}^{n} T_{k}\right)\right\|_{ L^ {p}_{\alpha}}^{\frac{2 p}{p-2 \alpha-d-2}}}=+\infty$$
\subsection{Blow-up result in $L_{\alpha}^{p}\left(\mathbb{R}_{+}^{d+1}\right), 2 \alpha+d+2<p \leq \infty$}
\begin{theorem}
Let $ \nu> 0,~ 2 \alpha+d+2 <p <\infty $ and $ u \in C \left([0, T^{*} [, \left(L^ {p} \left(\mathbb {R}^{d+1} \right) \right)^ {d+1} \right) \text {a maximum solution of}   $
$ (N S W) $ such that $ T ^ {*} <\infty, $ then there is a constant $ C> 0 $ such that
$$\|u(t)\|_{L^{p}_{\alpha}} \geq \frac{C}{\left(T^{*}-t\right)^{{\frac{2 p}{p-2 \alpha-d-2}}}}, \quad \forall t \in\left[0, T^{*}\right]$$
\end{theorem}
\begin{proof}
	Let $ T \in \left [0, T ^ {*} [\text {define by} \right. $
	 $$
	T = \sup \left\{t \in [0, T ^ {*} [; \sup_{0 \leq z \leq t}\| u (z) \|_{L ^ {p}} <2 \left\| u ^ {0} \right\|_{L ^ {p}_{\alpha} } \right\} 
	$$
		By combining proposition 2.2 .2 and the continuity of $ t \mapsto\| u (t) \|_{L^ {p}}, $ we obtain
		$$
		\| u (T) \|_{L ^ {p}_{\alpha}} = 2 \left\| u ^ {0} \right\|_{L ^ {p}_{\alpha}}
		$$
		For $ t \in [0, T] $ we have
		$$
		\| u (t) \| _ {L ^ {p}_{\alpha}} \leq \left\| u^{0} \right\|_{L ^ {p}_{\alpha}} + 4CT^{ \left( \frac {2p}{p-2\alpha-d-2} \right)}  \left\| u ^ {0} \right\|_{L ^ {p}_{\alpha}} ^ {2}
		$$
		In particular, for $ t = T $
		$$
		\left\| u ^ {0} \right\|_{L ^ {p}} \leq 4CT^{ \left( \frac {2p}{p-2\alpha-d-2} \right)}  \left\| u ^ {0} \right\|_{L ^ {p}_{\alpha}} ^ {2}
	$$
		so
		$$
		1 \leq 4CT^{ \left( \frac {2p}{p-2\alpha-d-2} \right)}  \left\| u ^ {0} \right\|_{L ^ {p}_{\alpha}} 
		$$
		Which give
		\begin{equation}\label{311}
		1 \leq 4CT^{* \left( \frac {2p}{p-2\alpha-d-2} \right)}  \left\| u ^ {0} \right\|_{L ^ {p}_{\alpha}} 
	\end{equation}
		Let $ t_ {0} \in \left[0, T ^ {*} [\right. $
\begin{equation*}\left( NSW_0\right)\left\{\begin{array}{l}
\partial_{t} v-\nu \Delta_{W}^{\alpha, d} v+\operatorname{div}_{\mathrm{W}}^{\alpha, d}(v \otimes v)=-\nabla_{W}^{\alpha, d} p,  \\
\operatorname{div}_{{W}}^{\alpha, d} v=0 \\
v(0)=u(t_0)
\end{array}\right.\end{equation*}
According to Theorem $\ref{th}, $ there is a single maximum solution $ v \in C \left ([0, A ^ {*}  [, \left(L ^ {p} \left(\mathbb {R} ^ {d+1} \right) \right) ^ {d+1} \right),
\\ A ^ {*} \in ]  0, + \infty] $.\\
Now $t \mapsto u\left(t+t_{0}\right)$ is a solution on $[0, T^{*}-t_{0}[, \text { then } $ 
$$v(t)=u\left(t+t_{0}\right) \text { and } A^{*}=T^{*}-t_{0}.$$
By applying \eqref{311} we obtain,
\begin{equation*}
1 \leq 4C(T^*-t)^{  \frac {2p}{p-2\alpha-d-2} }  \left\| v ( {0}) \right\|_{L ^ {p}_{\alpha}} 
\end{equation*}
By replacing the initial instant with any instant $t_0$, we deduce,
$$\left\| u ( {t_0}) \right\|_{L ^ {p}_{\alpha}} \geq \frac{1}{4 C\left(T^{*}-t_{0}\right)^{ \frac {2p}{p-2\alpha-d-2} }}$$
\end{proof}

\textbf{\underline{Question:} What are the conditions on the soboleve space $H_{\mathcal{S}_{*}}^{s, \alpha}\left(\mathbb{R}_{+}^{d+1}\right)$ (see \cite{444}) to define solutions of the Navier-Stokes equations associated with the Weinstein operator?}


\begin{thebibliography}{99}
	
	\bibitem {Aliev1}I. Aliev. Investigation on the Fourier-Bessel harmonic
	analusis, Doctoral Dissertation, Baku 1993 ( in Russian)
	\bibitem {Aliev3}I.A. Aliev and B. Rubin. Spherical harmonics associated to
	the Laplace-Bessel operator and generalized spherical convolutions. Anal.
	Appl. (Singap) Nr 1 (2003), p. 81-109.
\bibitem{12} J. Leray, Sur le mouvement d'un liquide visqueux emplissant l'espace, Acta Math. $63(1934),$ pp. $193-248$
\bibitem {bet}N. Bettaibi and H. Ben Mohamed, \textit{Sobolev type spaces
	associated with the Weinstein operator}, Int. Journal of Math. Analysis, Vol.5,Nr.28,(2011),p.1353-1373.
\bibitem{6} M. Brelot. Equation de Weinstein et potentiels de Marcel Riez. Lecture notes in Mathematics $681,$ Séminaire de théorie de Potentiel Paris $\mathrm{Nr} 3,(1978), \mathrm{p}$ $18-38$
\bibitem{1} Z. B. Nahia and N. B. Salem, "Spherical Harmonics and Applications Associated with the Weinstein Operator," $P o-$ tential Theory $1 \mathrm{CPT} 94,1996,$ pp. $235-241$
\bibitem{2} Z. B. Nahia and N. B. Salem, "On a Mean Value Property Associated with the Weinstein Operator," Potential Theory $1 C P T 94,1996,$ pp. $243-253$
\bibitem{444} H. Ben Mohamed, N. Bettaibi and S. H. JahInt. Sobolev type spaces associated with theWeinstein operator Journal of Math. Analysis, Vol. 5, 2011, no. 28, 1353 - 1373
\bibitem{mj1} H.MejjaoliJournal Hardy-type inequalities associated with
the Weinstein operator of Inequalities and Applications  (2015) 2015:267 
\bibitem{10} Y. Giga, Solutions for semilinear parabolic equations in Lp and regularity of weak solutions of the Navier-Stokes system, J. Differ. Equations $62(2),(1986),$ pp. $186-212$.
\bibitem{J} J. Benameur, On the blow-up criterion of $3 D$ Navier-Stokes equations, J. Math. Anal. Appl., 371 $(2010),$ pp. $719-727$
\bibitem {bn}Z. Ben\ Nahia, \textit{Fonctions harmoniques et propriet\'{e}s de
	la moyenne associ\'{e}es \`{a} l'op\'{e}rateur de Weinstein}, Th\`{e}se
3$^{\grave{e}me}$ cycle Maths. (1995) Department of Mathematics Faculty of
Sciences of Tunis. Tunisia.
\bibitem {Lof}J. L\"{o}fstr\"{o}m and J. Peetre. Approximation theorems
connected with generalized translations. Math. Ann. Nr 181 (1969), p. 255-268
\bibitem {Kip}I.A. Kipriyanov. Singular Elliptic Boundary Value Problems,
Nauka, Fizmatlit, 1997 (in Russian).
\bibitem {trim}K. Trim\`{e}che. Generalized Wavet and Hypergroups. Gordon and
Breach, New York, 1997.
\bibitem {Aliev2} B. Youssef., Ben Mohamed, H. \textit{Generalized Weinstein transform in quantum calculus} MathLAB Journal, 3, 50-65.
\bibitem {bn1}Z. Ben\ Nahia and N. BEN SALEM.\textbf{\ }Spherical harmonics
and applications associated with the Weinstein operator. `` Proceedings '' de
la Conf\'{e}rence Internationale de Th\'{e}orie de Potentiel, I.C.P.T. 94,
tenue \`{a} Kouty ( en R\'{e}publique Tch\`{e}que ) du 13-20 Ao\^{u}t 1994.
\bibitem {Lof}J. L\"{o}fstr\"{o}m and J. Peetre. Approximation theorems
connected with generalized translations. Math. Ann. Nr 181 (1969), p. 255-268
\end{thebibliography}
\end{document}